\documentclass[a4paper, 10pt]{article}

\input xy
\xyoption{all}

\newcounter{ExacSeq}

\usepackage[normalem]{ulem}
\usepackage{amsmath,amsthm,amssymb}
\usepackage{stmaryrd}
\usepackage[page,header]{appendix}
\usepackage{morefloats}
\usepackage{color}
\usepackage{makeidx}
\usepackage{fancybox}
\usepackage{mathabx}
\usepackage{tikz}
\usetikzlibrary{matrix}
\usepackage{arcs}
\usepackage{hyperref}
\usepackage{cancel}

\usepackage{yfonts}
\usepackage[T1]{fontenc}

\usepackage[normalem]{ulem}
\usepackage{amsmath}
\newcommand{\stkout}[1]{\ifmmode\text{\sout{\ensuremath{#1}}}\else\sout{#1}\fi}

\makeindex 

\usepackage{latexsym, amsmath, amssymb, amsfonts, amscd}
\usepackage{amsthm}
\usepackage{stmaryrd}
\usepackage{t1enc}
\usepackage[mathscr]{eucal}
\usepackage{indentfirst}
\usepackage{graphicx, pb-diagram}
\usepackage{enumerate}
\usepackage[all,poly,web,knot]{xy}
\usepackage{float}
\restylefloat{figure}
\usepackage{kantlipsum}
\usepackage[shortlabels]{enumitem}

\newcommand{\Title}{Title}

\numberwithin{equation}{section}

{\theoremstyle{definition}\newtheorem{definition}{Definition}[section]
	
	\newtheorem{defititle}[definition]{\Title}

	\newtheorem{remark}[definition]{Remark}
	\newtheorem{remarks}[definition]{Remarks}

	}

\newtheorem{prop}[definition]{Proposition}
\newtheorem{proposition-definition}[definition]{Proposition-Definition}
\newtheorem{lemma}[definition]{Lemma}
\newtheorem{thm}[definition]{Theorem}

\newtheorem*{prop*}{Proposition}
\newtheorem*{theorem*}{Theorem}

\newcommand{\cJ}{\mathcal{J}}

\newcommand{\Q}{\mathbb{Q}}

\newcommand{\id}{{\hbox{id}}}

\newcommand{\eg}{{\it e.g.}\/ }
\newcommand{\cf}{{\it cf.}\/ }

\def\gpd{\,\lower1pt\hbox{$\longrightarrow$}\hskip-.24in\raise2pt
	\hbox{$\longrightarrow$}\,}

\renewcommand{\latticebody}{\drop@{ }}

\newcommand{\Z}{\ensuremath{\mathbb Z}}
\newcommand{\C}{\ensuremath{\mathbb C}}
\newcommand{\R}{\ensuremath{\mathbb R}}

\newcommand{\T}{\ensuremath{\mathbb{T}}}

\newcommand{\cA}{\mathcal{A}}

\newcommand{\cX}{\mathcal{X}}

\DeclareMathOperator{\pr}{pr} 

\def\act{\mathbin{\hbox{$<\kern-.4em\mapstochar\kern.4em$}}}
\def\ract{\mathbin{\hbox{$\mapstochar\kern-.3em>$}}}

\def\PB(#1,#2,#3,#4){\left\{\begin{matrix}#1&\!\!\!\stackrel{?}{\longrightarrow}&\!\!\!#2\\
		\downarrow&&\!\!\!\downarrow\\
		#3&\!\!\!\stackrel{?}{\longrightarrow}&\!\!\!#4\end{matrix}\right\}}

\def\pb(#1,#2,#3,#4){ \hom(#1 \to #3, #2 \to #4)}

\begin{document}
	
	
	\begin{center}
		{\Large\bf On a remark by Alan Weinstein
			
			\bigskip
			
			{\sc by Iakovos Androulidakis}
		}
		
	\end{center}

	{\footnotesize
		\vskip -2pt National and Kapodistrian University of Athens
		\vskip -2pt Department of Mathematics
		\vskip -2pt Panepistimiopolis
		\vskip -2pt GR-15784 Athens, Greece
		\vskip -2pt e-mail: \texttt{iandroul@math.uoa.gr}
	}
	\bigskip
	

	\begin{abstract}
	Alan Weinstein remarked that, working in the framework of diffeology, a construction from Noncommutative Differential Geometry might provide the non-trivial representations required for the geometric quantisation of a symplectic structure which is not integral.  In this note we show that the construction we gave with P. Antonini does indeed provide non-trivial representations.
	\end{abstract}
	
	
	
	
	
	\setcounter{tocdepth}{2} 
	\tableofcontents
	
	\section*{Introduction}
	Let $(M,\omega)$ be a simply connected symplectic manifold. The first step towards geometric quantization (see \eg \cite{Souriau80}) is prequantization: The group  $\Gamma_{\omega}$ of periods of $[\omega] \in H^2_{dR}(M,\R)$ is a subgroup of the abelian group $(\R,+)$. When $\Gamma_{\omega}$ is cyclic (hence discrete), the quotient $G(M,\omega) = \R / \Gamma_{\omega}$ can be identified with the circle $U(1)$ and there exists a principal $U(1)$-bundle $Q$ over $M$ with a connection whose curvature is the original symplectic form $\omega$. Then, using a non-trivial representation of $G(M,\omega)$, one constructs a Hilbert space which is the crucial ingredient of the quantisation process. 
	
	Alan Weinstein pointed out in \cite[Rem. 5.1]{We89} that, when $\Gamma_{\omega}$ is a non-closed dense subgroup of $\R$, the construction of the Hilbert space is no longer possible. That is because the topology of the group $G(M,\omega)$ is trivial, which allows continuity only for the trivial representation. In the same remark, Weinstein proposed that some construction from noncommutative differential geometry might allow to ``quantize the unquantizable''. In other words, to obtain non-trivial representations. Note that when $\Gamma_{\omega}$ is not cyclic we can still make sense of the prequantizing bundle, using the notion of \textit{diffeology} which was introduced by Souriau \cite{Souriau70}. Actually, the results of \cite{We89} hold in this setting as well.
	
	The purpose of the present note is to confirm that, looking at the group $G(M,\omega)$ as a diffeological space, noncommutative differential geometry does indeed allow us to find non-trivial representations. The geometric construction suggested in \cite[Rem. 5.1]{We89} is given in \cite{AnAn18}: There we showed  that when the group $G(M,\omega)$ is not a circle, adding a few extra dimensions allows us to realise it as the leaf space of an irrational rotation foliation on a torus. As we will discuss here, this torus generates the atlas of bisubmersions of this foliation, in the sense of \cite{AndrSk}. These bisubmersions give rise to the diffeological structure of $G(M,\omega)$.
	
	What remains is to explain why this construction provides non-trivial representations. Recall that the representations ring of a Lie group can be recovered from the $K$-theory of its reduced $C^{\ast}$-algebra. Here we first explain how to make sense of $C^{*}(G(M,\omega))$ by applying the $C^{*}$-algebra construction we gave in \cite{AndrSk} to the diffeological structure generated by the torus. In order to show that the $K$-theory of this $C^{*}$-algebra is not trivial, we prove in Theorem \ref{thm} that $C^{*}(G(M,\omega))$ corresponds to the irrational rotation algebra. This isomorphism should be thought of as an analogue of Thom isomorphism.
	\bigbreak
	\ 
	\textbf{Acknowledgements} I would like to thank Alan Weinstein for bringing \cite{We89} to my attention. I would also like to thank Collegium Helveticum and the Institute of Mathematics of the University of Zurich, for their hospitality during the revision of this article.



\section{The Almeida-Molino groupoid as a diffeological groupoid}\label{sec:}

We will use the convenient setting of Lie algebroids and Lie groupoids throughout this note. In this section, first we recall this setting and then justify our construction of the convolution algebra $G(M,\omega)$ using diffeology.

\subsection{Integrality, Integrability and lifting}

Let us recall briefly how the integrality problem of the symplectic manifold $(M,\omega)$ can be viewed as an integrability problem for transitive Lie algebroids. This was first explained by Mackenzie in \cite{MK2}, see also \cite{Crainic04}. 

Consider the vector bundle $TM \oplus (M \times \R)$ over $M$. Since the 2-form $\omega$ is closed, the following Lie bracket on the $C^{\infty}(M)$-module of smooth sections of this bundle, satisfies the Jacobi identity: $$[X \oplus f,Y\oplus g] = [X,Y] \oplus \{X(g)-Y(f)-\omega(X,Y)\}$$ Here $X,Y \in \cX(M)$ and $f,g \in C^{\infty}(M)$. This bracket also satisfies the Leibniz identity with respect to the anchor map $X \oplus f \mapsto X$, therefore it defines a transitive Lie algebroid which we denote $TM \oplus_{\omega}(M \times \R)$. It integrates to a transitive Lie groupoid with isotropy group $U(1)$ if and only if the group $\Gamma_{\omega}$ of periods of $\omega$ is discrete. In this case, the group $G(M,\omega)$ is a Lie group; it is the isotropy group of the connected and simply connected Lie groupoid which integrates $TM \oplus_{\omega}(M \times \R)$, at an arbitrary point $x \in M$. The associated principal $U(1)$-bundle, together with the connection arising from the splitting $X \mapsto X \oplus 0$ of the anchor map, constitute the prequantization data of $(M,\omega)$.

The group $\Gamma_{\omega}$ of periods may not be closed though. For example, put $M = S^2 \times S^2$ and $\omega = p_1^{\ast}(V) \oplus \lambda p_2^{\ast}(V)$, where $V$ is a volume form of $S^2$, $p_i : M \to S_2$ the projection maps, $i=1,2$ and $\lambda$ an irrational real number. Then $\Gamma_{\omega} = \Z + \lambda\Z$ is a dense and non-closed subgroup of $\R$. Hence, the isotropy group $G(M,\omega) = \R / \Gamma_{\omega}$ has a quite pathological quotient topology. 

In fact, the Lie algebroid $TM \oplus_{\omega}(M \times \R)$ associated to this symplectic structure is the first example of a non-integrable (transitive) Lie algebroid; it was given by Almeida and Molino in \cite{AM}. As shown in \cite{AnAn18}, this example encompasses all the ``unquantizable'' (equivalently, non-integrable) symplectic structures.

Let us recall from \cite[Ex. 2.7]{AnAn18} a construction which, as we will explain in \S \ref{sec:naive}, allows us to view the Almeida-Molino example, within foliation theory and realise the topologically ill-behaved group $G(M,\omega)$ as a leaf space. This starts from observing that the reason for the pathology of $G(M,\omega)$ is really that the dimension of the vector space $\R$ is too small to accommodate an appropriate lattice. So, naively, we add one more dimension. Namely, we consider the $\R^2$-valued 2-form $\overline{\omega} = (p_1^{\ast}(V), \lambda p_2^{\ast}(V)) : TM \times TM \to M \times \R^2$. That is to say, the form defined by $$\overline{\omega}(X_{(x,y)},Y_{(x,y)})=((x,y),(p_1^{\ast}(V)(X_{(x,y)},Y_{(x,y)}),\lambda p_2^{\ast}(V)(X_{(x,y)},Y_{(x,y)}))$$ for all $(x,y) \in S^2 \times S^2$ and $X,Y \in \cX(S^2 \times S^2)$. Its group of periods $\Gamma_{\overline{\omega}}$ is the lattice $\Z \times \lambda\Z$ of $\R^2$, whence $TM \oplus_{\overline{\omega}}(M \times \R^2)$ is an integrable transitive Lie algebroid. The associated isotropy group $G(M,\overline{\omega}) = \R^2 / \Gamma_{\overline{\omega}}$ is diffeomorphic to the torus $\T^2$. Moreover, we have the ``addition'' map; that is to say the surjective morphism of Lie algebroids $$TM \oplus_{\overline{\omega}}(M \times \R^2) \to TM \oplus_{\omega}(M \times \R) \qquad (X,(a,b)) \mapsto (X,a+b)$$  for every $X \in TM$ and $a, b \in \R$. Integrating this map and restricting to the isotropy groups at an arbitrary point $x \in M$, we obtain the short exact sequence of continuous group morphisms 
\begin{eqnarray}\label{seq1}
0 \to \R \stackrel{\iota}{\longrightarrow} \mathbb{T}^2 \stackrel{\pi}{\longrightarrow} G(M,\omega) \to 1
\end{eqnarray}
The inclusion map $\iota$ is the irrational rotation with slope $\lambda \in \R \setminus \Q$ on the torus. Of course, in the exact sequence \eqref{seq1}, $(\R,+)$ and $\mathbb{T}^2$ are Lie groups and $\iota$ is a (smooth) morphism of Lie groups. 

In \cite[Thm. 3.6]{AnAn18} it was shown that the above lifting construction can be done systematically for an arbitrary simply connected\footnote{In \cite[\S 4]{AnAn18} it is shown that the simple connectivity assumption can be lifted, under a mild assumption.} symplectic manifold $(M,\omega)$ with finitely generated second fundamental group $\pi_2(M)$. Let us recall briefly from \cite[\S 3.1, 3.2]{AnAn18} the general construction of the integrable lift:
\begin{enumerate}
\item The familiar de Rham isomorphism identifies $Hom(H^2_{dR}(M),\R)$ with the bidual $H_2(M,\R)^{\ast\ast}$. Hence the natural map $$\cJ_M : H_2(M,\R) \to Hom(H^2_{dR}(M);\R), \cJ_M([\gamma])([\omega]) = \int_{\gamma} \omega$$ can be seen as an element $\cJ_M$ in $Hom(H_2(M,\R);H_2(M,\R)^{\ast\ast})$ and the latter can be identified with $Hom(H_2(M,\R);H_2(M,\R))$. The de Rham theorem implies that this element defines a class in de Rham cohomology $H^2_{dR}(M,H_2(M,\R))$ with coefficients in the homology vector space $H_2(M,\R)$. Let $\Theta$ in $\Omega^2(M,H_2(M,\R))$ be a closed 2-form which represents this class.

\item The transitive Lie algebroid $TM \oplus_{\Theta}(M \times H_2(M,\R))$ is integrable, thanks to the Hurewicz theorem and the assumption that $M$ is simply connected.

\item Now consider the closed 2-form $\Theta \oplus \omega$ in $\Omega^2(M,H_2(M,\R) \times \R)$. The previous integrability result implies the integrability of the transitive Lie algebroid $$TM \oplus_{\Theta\oplus\omega}(M \times (H_2(M,\R) \times \R))$$

\item The projection $p : H_2(M,\R) \times \R \to \R$ induces the surjective morphism of transitive Lie algebroids $$q_{\omega} = id_{TM}\oplus p : TM \oplus_{\Theta\oplus\omega}(M \times (H_2(M,\R) \times \R)) \longrightarrow TM \oplus_{\omega}(M \times \R)$$ with kernel the trivial bundle $M \times H_2(M,\R)$.
\end{enumerate}
Integrating the map $q_{\omega}$ and restricting to the isotropy group at an arbitrary point $x \in M$ we obtain a short exact sequence of topological groups
\begin{eqnarray}\label{seq2}
0 \to H_2(M,\R) \stackrel{\iota}{\longrightarrow} \mathbb{T}^{dim(H_2(M,\R)) + 1} \stackrel{\pi}{\longrightarrow} G(M,\omega) \to 1
\end{eqnarray}

\subsection{Beyond the naive $\ast$-algebra of $G(M,\omega)$}\label{sec:naive}

In sake of simplicity, throughout the sequel we focus on the exact sequence \eqref{seq1}. The injective homomorphism of Lie groups $$\iota_{\lambda} : (\R,+) \to (\mathbb{T}^2,\cdot) \quad t \mapsto (e^{2\pi i t},e^{2\pi i \lambda t})$$ induces a surjective linear map $(\iota_{\lambda})^{\ast} : C^{\infty}(\mathbb{T}^2) \to C^{\infty}(\R)$. 
Since the map $\pi : \mathbb{T}^2 \to G(M,\omega)$ is a surjective homomorphism of groups, it makes sense to define the naive function space $C^{\infty}(G(M,\omega))$ to be $\ker(\iota_{\lambda})^{\ast}$. We might then consider restricting the $\ast$-algebra structure of $C^{\infty}(\mathbb{T}^2)$ induced by the group structure of $\mathbb{\T}^2$ to $C^{\infty}(G(M,\omega))$. However, as it happens, the $\ast$-algebra arising this way does not carry any information\footnote{Proposition \ref{prop:naive} shows that $C^{\infty}(\mathbb{T}^2)$ and $C^{\infty}(\R)$ are isomorphic only as vector spaces. However, they cannot be isomorphic as topological vector spaces.}.
\begin{prop}\label{prop:naive}
The ideal $\ker(\iota_{\lambda})^{\ast}$ is trivial.
\end{prop}
\begin{proof}
By the diffeomorphism $\mathbb{T}^2 = S^1 \times S^1$ we have $\iota_{\lambda}(\R) = \{(e^{2\pi i t}, e^{2\pi i \lambda t}) : t \in \R\}$. Since $\lambda \in \R / \Q$, this set is dense in $\mathbb{T}^2$. Now the $\ker(\iota_{\lambda})^{\ast}$ is the ideal of smooth functions on  $\mathbb{T}^2$ which vanish on $\iota_{\lambda}(\R)$. Continuity implies that such a function vanishes everywhere in $\mathbb{T}^2$.
\end{proof}

Proposition \ref{prop:naive} shows that the group structure of the exact sequence \eqref{seq1} is not the correct way to attach a $C^{\ast}$-algebra to the group $G(M,\omega)$. We really need to change our point view. 

A second reading of the group homomorphism $\pi : \mathbb{T}^2 \to G(M,\omega)$ in \eqref{seq1},  is that it provides $G(M,\omega)$ with the structure of a diffeological space in the sense of \cite{Souriau70} \cite{Souriau80}. Using this structure, in \S \ref{sec:convolutionalgebra} we will attach a space of compactly supported functions $C^{\infty}_c(G(M,\omega))$ to $G(M,\omega)$.

However, there is yet another way to read \eqref{seq1}: It really describes the foliation on $\mathbb{T}^2$ by irrational rotation, namely the foliation by the orbits of the following action: $$\mathbb{T}^2 \times \R \to \mathbb{T}^2 \quad ((e^{2\pi i x},e^{2\pi i \lambda y}),t) \mapsto (e^{2\pi i (x+t)},e^{2\pi i (y+\lambda t)}) = (e^{2\pi i x},e^{2\pi i \lambda y}) \cdot \iota_{\lambda}(t)$$ To see this, fix a point $p = (e^{2\pi i x},e^{2\pi i \lambda y})$ in $\mathbb{T}^2$. Then the submanifold $\iota_{\lambda}(\R)$ can be identified with the orbit of $p$ by the irrational rotation action. This implies that the quotient $G(M,\omega) = \mathbb{T}^2 / \iota_{\lambda}(\R)$ arising from \eqref{seq1} can be identified with the orbit space of this action. In other words, the leaf space\footnote{Remarkably, extension \eqref{seq1} shows that this leaf space carries a group structure as well.} of the irrational rotation foliation.   

On the other hand, in foliation theory, it is well known that the leaf space can always be described by symmetries, both internal and external: Recall that, given a foliation $(M,F)$, a model for the leaf space $M/F$ is the holonomy groupoid $H(M,F)$. In the case of the irrational rotation on the torus, the holonomy groupoid is the action groupoid $\mathbb{T}^2 \rtimes_{\lambda} \R \gpd \mathbb{T}^2$ associated with the previous action. Let us recall the structure of this groupoid. For every $g = (e^{2\pi i \theta}, e^{2\pi i \eta}) \in \mathbb{T}^2$ and $t \in \R$ we have: $$s(g,t)=g, \quad r(g,t)=gt, \quad (g_1,t) (g_2,s)=(g_2,s + t), \quad (g,t)^{-1} = (gt,-t)$$

Also recall that the $C^{\ast}$-algebra $\mathbb{T}^2 \rtimes_{\lambda} \R$ is the crossed product\footnote{Recall that the full crossed product $C^{\ast}$-algebra associated with $\mathbb{T}^2 \rtimes_{\lambda} \R$ is isomorphic with the reduced one.} $C(\mathbb{T}^2) \rtimes_{\lambda} \Z$. It is well-known that this crossed product is the irrational rotation algebra $A_{\lambda}$, namely the universal $C^{\ast}$-algebra generated by two unitaries $u, v$ such that $uv = e^{2\pi i \lambda} vu$.

The purpose of this sequel is to prove the following result:
\begin{thm}\label{thm}
The space $C^{\infty}_c(G(M,\omega))$ admits a completion to a $C^{\ast}$-algebra, which is isomorphic to the irrational rotation $C^{\ast}$-algebra $A_{\lambda}$.
\end{thm}

\section{The function space of $G(M,\omega)$}\label{sec:convolutionalgebra}

In this section, we attach a $\C$-linear space of smooth, compactly supported densities on the group $G(M,\omega)$. To this end, we use the diffeological structure of $G(M,\omega)$, which we describe in \S \ref{sec:diffeoG}. Our construction of this function space in \S \ref{sec:convolutionalgebraconstruction} adapts to the diffeology of $G(M,\Omega)$ the construction of the convolution algebra of the holonomy groupoid of a foliation, given in \cite{AndrSk}.

\subsection{The diffeology of $G(M,\omega)$}\label{sec:diffeoG}

Consider the following family of plots:
\begin{eqnarray*}
P(G(M,\omega)) =  \{ \pi|_{\mathcal{O}} : \mathcal{O} \to G(M,\omega), \text{ where }\mathcal{O} \text{ is an open subset of }\mathbb{T}^2\}
\end{eqnarray*}
Of course, here ``open'' is understood with respect to the well known differential structure of the torus. Whence, each $\mathcal{O}$ may also be thought of as an open subset of $\R^2$. 

Now put $D(G(M,\omega))$ the diffeology generated by $P(G(M,\omega))$, namely the smallest diffeology containing these plots. It is characterised as follows: For $k =1,2,\ldots$, a $k$-plot $\chi : \mathcal{O}_{\chi} \to G(M,\omega)$ belongs to $D(G(M,\omega))$ if and only if for every point $x \in \mathcal{O}_{\chi}$ there is a neighbourhood $\mathcal{O}_x$ of $x$ such that either $\chi|_{\mathcal{O}_x}$ is constant, or $\chi|_{\mathcal{O}_x} = \pi|_{\mathcal{O}} \circ h$, where $\pi|_{\mathcal{O}} : \mathcal{O} \to G(M,\omega)$ is a plot in $P(G(M,\omega))$ and $h : \mathcal{O}_x \to \mathcal{O}$ is a smooth map. In remarks \ref{rem:cartprod} we give a few properties of this diffeology.

\begin{remarks}\label{rem:cartprod}
\begin{enumerate}
\item Recall that the diffeology $D(G(M,\omega))$ induces a topology on $G(M,\omega)$ (``$D$-topology''). A subset $U$ of $G(M,\omega)$ is open in this topology if and only if for each plot $\chi$ in $P(G(M,\omega))$ the inverse image $\chi^{-1}(W)$ is open in $\mathcal{O}_{\chi}$. It is the same as the quotient topology defined from \eqref{seq1}.

\item We may assume that the diffeology $D(G(M,\omega))$ is generated by a finite number of open sets. Indeed, using compactness, let us fix a finite open cover $\mathcal{O}_1,\ldots,\mathcal{O}_n$ of $\mathbb{T}^2$. For every $i=1,\ldots,n$ put $\pi_i = \pi|_{\mathcal{O}_i}$ and define $P(G(M,\omega)) = \{\pi_i : \mathcal{O}_i \to G(M,\omega) : i = 1,\ldots,n\}$. Let $\mathcal{O}$ be an arbitrary open subset of $\mathbb{T}^2$ and $x$ a point in $\mathcal{O}$. Since $\bigcup_{i=1}^{n}\mathcal{O}_i = \mathbb{T}^2$, $x$ belongs to $\mathcal{O}_i$ for some $1 \leq i \leq n$. Since $\mathcal{O}$ is open, there is a small enough neighbourhood $\mathcal{O}_x$ of $x$ in $\mathcal{O}$ such that the identity $h = \id_{\mathcal{O}_x}$ is a map $h : \mathcal{O}_x \to \mathcal{O}_i$. We obviously have $\pi|_{\mathcal{O}_x} =  \pi_i \circ h$. Whence, for every open subset $\mathcal{O}$, the plot $\pi|_{\mathcal{O}} \to G(M,\omega)$ belongs to $D(G(M,\omega))$. Moreover, $D(G(M,\omega))$ does not depend on the choice of finite cover of $\mathbb{T}^2$.

\item Consider the family of plots $P(G(M,\omega)) = \{(\pi|_{\mathcal{O}_i},\mathcal{O}_i)\}_{i \in I}$, where $\mathcal{O}_i$ is an open subset of $\mathbb{T}^2$. Put $\mathcal{O} = \coprod_{i \in I}\mathcal{O}_{i}$ and let $\overline{\pi} : \mathcal{O} \to G(M,\omega)$ be the map whose restriction to each $\mathcal{O}_{i}$ is $\pi|_{\mathcal{O}_i}$. Obviously, $(\overline{\pi},\mathcal{O})$ is a plot in $D(G(M,\omega))$.

\item If $\mathcal{O}_i$ are plots in $P(G(M,\omega))$ for $i=1,2$, then the cartesian product $\mathcal{O}_1 \times \mathcal{O}_2$ is a plot in the diffeology $D(G(M,\omega))$. Indeed, define $\chi : \mathcal{O}_1 \times \mathcal{O}_2 \to G(M,\omega)$ by $\chi=m_G \circ (\pi|_{\mathcal{O}_1} \times \pi|_{\mathcal{O}_2})$, where $m_G$ is the product map in $G(M,\omega)$. Now put $h=m_{\mathbb{T}^2}|_{\mathcal{O}_1 \times \mathcal{O}_2} : \mathcal{O}_1 \times \mathcal{O}_2 \to \mathbb{T}^2$ the restriction of the product map of the Lie group $\mathbb{T}^2$. Since $m_{\mathbb{T}^2} : \mathbb{T}^2 \times \mathbb{T}^2 \to \mathbb{T}^2$ is a submersion, the image of $h$ is an open subset $\mathcal{O}$ of $\mathbb{T}^2$. The map $\pi$ is a morphism of groups, therefore we have $\pi|_{\mathcal{O}} \circ h = \chi$.


\item Let $\iota_{\mathbb{T}^2} : \mathbb{T}^2 \to \mathbb{T}^2$ be the inversion map. For every plot $\mathcal{O}$ in $P(G(M,\omega))$, the open subset $\iota(\mathcal{O}) \subseteq \mathbb{T}^2$ is also a plot. Put $\chi : \iota(\mathcal{O}) \to G(M,\omega)$ the map $\pi|_{\mathcal{O}} \circ \iota$ and $h = \iota : \iota(\mathcal{O}) \to \mathcal{O}$. We have $\chi = \pi|_{\mathcal{O}} \circ \iota$.

\item It follows that the diffeology $D(G(M,\omega))$ is closed by finite cartesian products and inverses of plots in $P(G(M,\omega))$.

\item Let $\chi_i : \mathcal{O}_{\chi_i} \to G(M,\omega)$, $i=1,2$ be two plots in $D(G(M,\omega))$. Put $\chi = m_G \circ (\chi_1 \times \chi_2) : \mathcal{O}_{\chi_1} \times \mathcal{O}_{\chi_2} \to G(M,\omega)$. We call this plot the \textit{product} of the plots $(\chi_i,\mathcal{O}_i)_{i=1,2}$. Let us show that this is a plot in $D(G(M,\omega))$. Indeed, let $\mathcal{O}_i$ open subsets of $\mathbb{T}^2$ and $h_i : \mathcal{O}_{\chi_i} \to \mathcal{O}_{i}$ smooth maps such that $\chi_i = \pi|_{\mathcal{O}_i} \circ h$. Since $\pi$ is a morphism of groups we have $\pi \circ m_{\mathbb{T}^2} = m_G \circ (\pi \times\ldots\times\pi)$. This implies that $\chi = \pi|_{\mathcal{O}_{12}} \circ (h_1 \times h_2)$, where $\mathcal{O}_{12}$ is the image of $\mathcal{O}_1 \times \mathcal{O}_2$ by the product map $m_{\mathbb{T}^2}$.

\item In view of the previous item, denote $\mathcal{O}_{\chi_{12}}$ the image of the restriction of product map $m_{\mathbb{T}^2}$ to $\mathcal{O}_{\chi_1} \times \mathcal{O}_{\chi_2}$. Since $m_{\mathbb{T}^2}$ is a submersion, $\mathcal{O}_{\chi_{12}}$ is an open subset of $\mathbb{T}^2$. Choose a local section $\sigma : \mathcal{O}_{\chi_{12}} \to \mathcal{O}_{\chi_1} \times \mathcal{O}_{\chi_2}$ (shrinking $\mathcal{O}_{\chi_{12}}$ if necessary). Define $\chi_{12} : \mathcal{O}_{\chi_{12}} \to G(M,\omega)$ by $\chi_{12} = m_G \circ (\chi_1 \times \chi_2) \circ \sigma = \chi \circ \sigma = \pi \circ ((h_1 \times h_2)\circ\sigma)$. Whence $(\chi_{12},\mathcal{O}_{\chi_{12}})$ is a plot in $D(G(M,\omega))$.

\item Let $\chi : \mathcal{O}_{\chi} \to G(M,\omega)$ be a plot in $D(G(M,\omega))$. Put $\chi_{-1} = \iota_G \circ \chi$, where $\iota_G$ is the inversion map of $G(M,\omega)$. We call this plot the \textit{inverse} of the plot $(\chi,\mathcal{O}_{\chi})$. Let $\mathcal{O}$ be an open subset of $\mathbb{T}^2$ and $h : \mathcal{O}_{\chi} \to \mathcal{O}$ a smooth map such that $\chi = \pi|_{\mathcal{O}} \circ h$. The identity $\iota_G \circ \pi = \pi \circ \iota_{\mathbb{T}^2}$ implies that $\chi_{-1} = \pi|_{\iota_{\mathbb{T}^2}(\mathcal{O})} \circ (\iota_{\mathbb{T}^2} \circ h)$. 

\item Alternatively, we can define the inverse of the plot $(\chi,\mathcal{O}_{\chi})$ to be $(\chi^{-1},\mathcal{O}_{\chi^{-1}})$, where $\mathcal{O}_{\chi^{-1}} = \iota_{\mathbb{T}^2}(\mathcal{O}_{\chi})$ and $\chi^{-1}=\chi\circ \iota^{-1}_{\mathbb{T}^2} = \pi|_{\mathcal{O}} \circ (h \circ \iota^{-1}_{\mathbb{T}^2})$. The pairs $(\chi_{-1},\mathcal{O}_{\chi})$ and $(\chi^{-1},\mathcal{O}_{\chi^{-1}})$ are equivalent, because $\iota_{\mathbb{T}^2} : \mathcal{O}_{\chi} \to \mathcal{O}_{\chi^{-1}}$ is a diffeomorphism and $\chi_{-1} = \iota_G \circ \chi^{-1} \circ \iota_{\mathbb{T}^2}$.

\item Items (g), (i) and (j) imply that $D(G(M,\omega))$ is closed by finite cartesian products and inverses of plots in $D(G(M,\omega))$. \end{enumerate}
\end{remarks}

Now let us denote $\{\mathcal{O}_i\}_{i \in I}$ the manifold atlas of $\mathbb{T}^2$. Put $\mathcal{O}$ the disjoint union $\coprod_{i \in I}\mathcal{O}_i$ and  $\chi : \mathcal{O} \to G(M,\omega)$ the map whose restriction to each $\mathcal{O}_i$ is $\pi|_{\mathcal{O}_i}$. It follows that $(\chi,\mathcal{O})$ is a plot in $D(G(M,\omega))$.

\begin{lemma}\label{lem1}
Let $\chi : \mathcal{O}_{\chi} \to G(M,\omega)$ be a plot in $D(G(M,\omega))$ and $x_0 \in \mathcal{O}_{\chi}$. There exists a plot $\psi : \mathcal{O}_{\psi} \to G(M,\omega)$ in $P(G(M,\omega))$ and submersions $p : \mathcal{O}_{\psi}\to \mathcal{O}$, $q :  \mathcal{O}_{\psi} \to  \mathcal{O}_{\chi}$ such that $x_0 \in q(\mathcal{O}_{\psi})$. Moreover, $p$ and $q$ satisfy $\pi|_{\mathcal{O}_i} \circ p = \psi = \chi \circ q$ for every $i \in I$.
\end{lemma}
\begin{proof}
From the definition of the diffeology $D(G(M,\omega))$ there is a neighbourhood $\mathcal{O}_{x_0}$ of $x_0$ in $\mathcal{O}_{\chi}$, an $i \in I$ and a smooth map $h : \mathcal{O}_{x_0} \to \mathcal{O}_i$ such that $\chi|_{\mathcal{O}_{x_0}}=\pi|_{\mathcal{O}_i}\circ h$. Put $\mathcal{O}_{\psi} = \{(y,p) \in \mathcal{O}_{x_0} \times \mathcal{O}_{i} : \chi(y)=\pi(p)\}$ and define $\psi : \mathcal{O}_{\psi} \to G(M,\omega)$ by $\psi(y,p)=\chi(y)=\pi(p)$. Endowing $G(M,\omega)$ with the $D$-topology, we find that $\mathcal{O}_{\psi}$ is an open subset of $\mathcal{O}_{x_0} \times \mathcal{O}_{i}$. 
The maps $p$ and $q$ are the respective projections. In particular, by definition the map $p : \mathcal{O}_{\psi} \to \mathcal{O}_i$ satisfies $\psi = \pi|_{\mathcal{O}_i}\circ p$, whence $(\psi,\mathcal{O}_{\psi})$ is a plot in $D(G(M,\omega))$. The equality $\psi = \chi \circ q$ follows from the definition of $\psi$.
\end{proof}

\begin{remark}
The construction of the plot $(\psi,\mathcal{O}_{\psi})$ implies $\psi = \pi|_{\mathcal{O}_i} \circ (h\circ q)$.
\end{remark}


\subsection{Construction of $C^{\infty}_c(G(M,\omega))$}\label{sec:convolutionalgebraconstruction}

In this spirit, let us introduce the following spaces of functions:
\begin{enumerate}
\item For every plot $\chi : \mathcal{O}_{\chi} \to G(M,\omega)$ in $D(G(M,\omega))$ consider the bundle of 1-densities $\Omega^1 (T\mathcal{O}_{\chi})$ over $\mathcal{O}_{\chi}$. Recall that this is a line bundle which becomes trivial once we choose one of its sections. We denote $C^{\infty}_{c}(\mathcal{O}_{\chi};\Omega^1 \mathcal{O}_{\chi})$ the associated $C^{\infty}(\mathcal{O}_{\chi})$-module of smooth, compactly supported sections. 



\item Let $\phi : X \to \mathcal{O}_{\chi}$ be a surjective submersion. The short exact sequence $$0 \to \ker d\phi \to TX \stackrel{d\phi}{\longrightarrow} T\mathcal{O}_{\chi} \to 0$$ implies $\Omega^1(TX) = \Omega^1(\ker d\phi) \otimes \phi^{\ast}(\Omega^1 (T\mathcal{O}_{\chi}))$. Integration along the fibers is a surjective map $$\phi_{!} : C^{\infty}_c(X;\Omega^1(TX)) \to C^{\infty}_c(\mathcal{O}_{\chi}; \Omega^1 (T\mathcal{O}_{\chi})) \qquad \phi_{!}(f) : x \mapsto \int_{\phi^{-1}(x)} f$$ Note that $C^{\infty}_{c}(X;\Omega^1 X)$ is a $C^{\infty}(\mathcal{O}_{\chi})$-module with $f\zeta = (f \circ \phi)\cdot\zeta$ for every $f \in C^{\infty}(\mathcal{O}_{\chi})$ and $\zeta \in C^{\infty}_{c}(X;\Omega^1 X)$.

\end{enumerate}

Recall the plot $(\overline{\pi},\mathcal{O})$ in $D(G(M,\omega))$ which arises from the disjoint union of plots $(\pi|_{\mathcal{O}_i},\mathcal{O}_i)$ in $P(G(M,\omega))$, mentioned in item (c) of remark \ref{rem:cartprod}.


\begin{lemma}\label{lem2}
Let $\chi : \mathcal{O}_{\chi} \to G(M,\omega)$ be a plot in $D(G(M,\omega))$ and $f$ in $C^{\infty}_c(\mathcal{O}_{\chi},\Omega^1\mathcal{O}_{\chi})$. Then there exist a plot $\psi : \mathcal{O}_{\psi} \to G(M,\omega)$ in $D(G(M,\omega))$, submersions $p : \mathcal{O}_{\psi}\to \mathcal{O}$, $q :  \mathcal{O}_{\psi} \to  \mathcal{O}_{\chi}$ and $g \in C^{\infty}(\mathcal{O}_{\psi},\Omega^1\mathcal{O}_{\psi})$ such that $q_{!}(g)=f$. 
\end{lemma}
\begin{proof}
By lemma \ref{lem1} and the compactness of the support of $f$ we find a finite number of plots $(\psi_j, \mathcal{O}_{\psi_j})$ in $D(G(M,\omega))$ and submersions $p_j : \mathcal{O}_{\psi_j} \to \mathcal{O}$ and $q_j : \mathcal{O}_{\psi_j} \to \mathcal{O}_{\chi}$ such that $\bigcup_{j} q_j(\mathcal{O}_{\psi_j})$ covers the support of $f$.  Put $\mathcal{O}_{\psi}$ the disjoint union $\coprod_{j} \mathcal{O}_{\psi_j}$ and $\psi = \coprod_j \psi_j$. It is easy to see that $(\psi,\mathcal{O}_{\psi})$ is a plot in $D(G(M,\omega))$. The maps $q = \coprod_{j}p_j : \mathcal{O}_{\psi} \to \mathcal{O}_{\chi}$ and $p = \coprod_{j}p_j : \mathcal{O}_{\psi} \to \mathcal{O}$ are submersions. From item (b) above, we have that $f$ is of the form $q_{!}(g)$.
\end{proof}

Using lemma \ref{lem2} we can define our convolution algebra. 

\begin{definition}\label{dfn1}
\begin{enumerate}
\item We define $C^{\infty}_c(G(M,\omega))$ to be the quotient $$\frac{\oplus_{i \in I}C^{\infty}_c(\mathcal{O}_i;\Omega^{1}(\mathcal{O}_i))}{\mathcal{I}}$$ where $\mathcal{I}$ is the subspace spanned by $p_!(f)$ where $p : \mathcal{O}_{\psi} \to \mathcal{O}$ is a submersion such that $q_!(f)=0$. Here the data $\mathcal{O}_{\psi}, p, q$ is as in lemma \ref{lem1}.

\item Let $\chi : \mathcal{O}_{\chi} \to G(M,\omega)$ be a plot in $D(G(M,\omega))$. With the notation of lemma \ref{lem2} we define a linear map $Q_{\chi} : C^{\infty}_c(\mathcal{O}_{\chi};\Omega^1 \mathcal{O}_{\chi}) \to C^{\infty}_c(G(M,\omega))$. It is defined by putting $Q_\chi(f)$ to be the class of $p_!(g) \in \oplus_{i \in I}C^{\infty}_c(\mathcal{O}_i;\Omega^{1}(\mathcal{O}_i))$ in the above quotient. Again, the data $\mathcal{O}_{\psi}, p, q, g$ is as in Lemma \ref{lem2}, in particular $f=q_{!}(g)$.
\end{enumerate}
\end{definition}

\begin{remarks}
\begin{enumerate}
\item Item (b) in remark \ref{rem:cartprod} implies that, in item (a) of definition \ref{dfn1} we can take the direct sum to be finite.
\item The map $Q_{\chi}$ is well defined. Indeed, for $\ell = 1,2$ consider plots $\mathcal{O}_{\psi_{\ell}}$, submersions $p_{\ell} : \mathcal{O}_{\psi_{\ell}} \to \mathcal{O}$, $q_{\ell} : \mathcal(O)_{\psi_{\ell}} \to \mathcal{O}_{\chi}$ and $g_{\ell} \in C^{\infty}_c(\mathcal(O)_{\psi_{\ell}};\Omega^1(\mathcal(O)_{\psi_{\ell}}))$ such that $(q_1)_{!}(g_1) = (q_2)_{!}(g_2)$. Put $\mathcal{O}_{\psi} = \mathcal{O}_{\psi_1} \coprod \mathcal{O}_{\psi_2}$, $q = q_1 \coprod q_2$ and $p = p_1 \coprod p_2$. The density $g = g_1 \coprod (-g_2) \in C^{\infty}_c(\mathcal{O}_{\psi};\Omega^1(\mathcal{O}_{\psi}))$ obviously satisfies $q_{!}(g)=0$ and $p_{!}(g)= (p_{1})_{!}(g_1) - (p_2)_{!}(g_2)$.
\end{enumerate}
\end{remarks}

The following characterization of the map $Q_{\chi}$ is proven exactly as in \cite[Prop. 4.4]{AndrSk}.

\begin{prop}\label{prop1}
To every plot $\chi : \mathcal{O}_{\chi} \to G(M,\omega)$ in $D(G(M,\omega))$ we can associate a linear map $Q_{\chi} : C^{\infty}_{c}(\mathcal{O}_{\chi};\Omega^1 \mathcal{O}_{\chi}) \to C^{\infty}_c(G(M,\omega))$. The maps $Q_{\chi}$ are characterised by the following properties:
\begin{enumerate}
\item If $\mathcal{O}_{\chi}=\mathcal{O}_i$, $Q_{\chi}$ is the quotient map $$C^{\infty}_{c}(\mathcal{O}_i;\Omega^1 \mathcal{O}_i) \subset \oplus_{i \in I}C^{\infty}_c(\mathcal{O}_i;\Omega^{1}(\mathcal{O}_i)) \to \frac{\oplus_{i \in I}C^{\infty}_c(\mathcal{O}_i;\Omega^{1}(\mathcal{O}_i))}{\mathcal{I}}$$ We use the notation $Q_i$ instead of $Q_{\chi}$ in this case.
\item For every smooth map $q : \mathcal{O}_{\psi} \to \mathcal{O}_{\chi}$ which is a submersion such that $\chi \circ q =\psi$, we have $Q_{\psi}=Q_{\chi}\circ q_!$.
\end{enumerate}
\end{prop}

\begin{remarks}\label{rmk:staralgebra}
\begin{enumerate}
\item In section \ref{sec:isomorphism} we provide a bijection between $C^{\infty}_c(G(M,\omega))$ and the convolution algebra of the action groupoid $\mathbb{T}^2 \rtimes_{\lambda} \R \gpd \mathbb{T}^2$. This $\ast$-algebra structure is defined using the groupoid structure of the manifold $\mathbb{T}^2 \rtimes_{\lambda} \R \gpd \mathbb{T}^2$, rather than the group structure. Hence, this bijection identifies the completion of $C^{\infty}_c(G(M,\omega))$ with the irrational rotation algebra $A_{\lambda}$; This is the algebra that we want to attach to the group $G(M,\omega)$, since it is the leaf space of the foliation by irrational rotation on $\mathbb{T}^2$. 

\item Using the group structures of $\mathbb{T}^2$ and $G(M,\omega)$, it is possible to endow $C^{\infty}_c(G(M,\omega))$ with a different $\ast$-algebra structure. We provide the explicit construction of this structure in Appendix \ref{app:convolutionalgebraproperties}. 

\item The bijection we define in \S \ref{sec:isomorphism} also conveys the $\ast$-algebra structure arising from the group structure of $\mathbb{T}^2 \rtimes_{\lambda} \R$ to the $\ast$-algebra structure of $C^{\infty}_c(G(M,\omega))$ described in Appendix \ref{app:convolutionalgebraproperties}. This is explained in Appendix \ref{app:convolutionalgebraproperties}, in case the reader finds it useful.
\end{enumerate}
\end{remarks}

\section{Isomorphism with irrational rotation algebra}\label{sec:isomorphism}



The purpose of this section is to prove Theorem \ref{thm}. To this end, in \S \ref{sec:transfgpd} we describe the diffeological structure of the action groupoid $\mathbb{T}^2 \rtimes_{\lambda} \R \gpd \mathbb{T}^2$ and in \S \ref{sec:convalggpd} we construct the associated $\ast$-algebra. The proof of Theorem \ref{thm} is given in \S \ref{sec:proof}.

\subsection{The transformation groupoid of the irrational rotation}\label{sec:transfgpd}



Let us recall the following:
\begin{enumerate}
\item In \cite[Examples 3.4, item 4]{AndrSk} it is shown that, for any $p \in \mathbb{T}^2$, a neighborhood of the identity element $(p,0)$ in $\mathbb{T}^2 \rtimes_{\lambda} \R$ is diffeomorphic to $\mathcal{O}_i \times (-\epsilon_{ij},\epsilon_{ij})$, for some open subset $\mathcal{O}_i$ of $\mathbb{T}^2$. 

\item Under this diffeomorphism, the source map becomes the projection to $\mathcal{O}_i$. The range map takes the form of the exponential $r(p,t) = exp(tX)(p)$, where $X$ is the infinitesimal generator of the irrational rotation action; it is a submersion $r : \mathcal{O}_i \times (-\epsilon_{ij},\epsilon_{ij}) \to \mathcal{O}_j$, where $\mathcal{O}_j$ is some other open subset of $\mathbb{T}^2$. 

\item Put $\phi_{ij} : \mathcal{O}_i \times (-\epsilon_{ij},\epsilon_{ij}) \to s^{-1}(\mathcal{O}_i) \cap r^{-1}(\mathcal{O}_j)$ this diffeomorphism. Then $(\mathcal{O}_i \times (-\epsilon_{ij},\epsilon_{ij}), \phi_{ij}, \mathbb{T}^2 \rtimes_{\lambda} \R)$ is a minimal bisubmersion\footnote{Here, the (``singular'') subalgebroid is the entire Lie algebroid of $\mathbb{T}^2 \rtimes_{\lambda} \R$. We put \textit{singular} in quotes because, in fact, there are no singularities: The module of sections of any Lie algebroid is projective, thanks to the Serre-Swan theorem.} in the sense of \cite[Definitions 2.4, 2.6]{Za22}. These are the path-holonomy bisubmersions of the Lie algebroid $A(\mathbb{T}^2 \rtimes_{\lambda} \R)$.

\item Put $\mathcal{A}$ the atlas of bisubmersions generated by the path-holonomy bisubmersions (\cf \cite[Dfn. 3.4]{Za22}). The Lie groupoid $\mathbb{T}^2 \rtimes_{\lambda} \R \gpd \mathbb{T}^2$ is the quotient groupoid associated to this atlas (\cf \cite[\S 3.2]{Za22}).

\item It follows from the previous item that the $C^*$-algebra of the atlas $\cA$ constructed in \cite[Appendix A]{Za22} is isomorphic with the crossed product $C(\mathbb{T}^2) \rtimes_{\lambda} \Z$, in other words with $A_{\lambda}$.
\end{enumerate}

Since $C^{\infty}_c(\mathbb{T}^2 \rtimes_{\lambda} \R)$ is $\ast$-isomorphic with the $\ast$-algebra associated with the atlas $\cA$ (\cf \cite[Appendix A]{Za22}), we need to understand the relation of bisubmersions as above with plots in $D(G(M,\omega))$. We explain this relation in the following lemmas.

\begin{lemma}\label{lem3}
For every $i \in I$ the path-holonomy bisubmersion $\mathcal{O}_i \times (-\epsilon_{ij},\epsilon_{ij})$ is a plot in $D(G(M,\omega))$.
\end{lemma}
\begin{proof}
Consider the submersion $r : \mathcal{O}_i \times (-\epsilon_{ij},\epsilon_{ij}) \to \mathcal{O}_j$ and define $\chi_{ij} = \pi|_{\mathcal{O}_j} \circ r$. Obviously, $(\chi_{ij}, \mathcal{O}_i \times (-\epsilon_{ij},\epsilon_{ij}))$ is a plot in $D(G(M,\omega))$.
\end{proof}

\begin{lemma}\label{lem4}
Every bisubmersion $(U,\phi_U, \mathbb{T}^2 \rtimes_{\lambda} \R)$ in the atlas $\mathcal{A}$ is a plot in $D(G(M,\omega))$. 
\end{lemma}
\begin{proof}
Up to shrinking $U$ if necessary, we can assume that $\phi(U)$ is contained in $\mathcal{O}_i \times (-\epsilon_{ij},\epsilon_{ij})$, for some $i,j \in I$. The map $\chi_U : U \to G(M,\omega)$ defined by $\chi = \pi|_{\mathcal{O}_j} \circ r \circ \phi$ makes $(\chi_U,U)$ a plot in $D(G(M,\omega))$.
\end{proof}

\begin{lemma}\label{lem5}
Let $(\chi,\mathcal{O}_{\chi})$ be a plot in $D(G(M,\omega))$ and $p : \mathcal{O}_{\chi} \to \mathcal{O}_{i}$ a submersion for some $i \in I$. Then $(\mathcal{O}_{\chi} \times (-\epsilon_{ij},\epsilon_{ij}), (p\times\id) \circ \phi_{ij}, \mathbb{T}^2 \rtimes_{\lambda} \R)$ is a bisubmersion for (the ``singular subalgebroid'') $\Gamma_c A(\mathbb{T}^2 \rtimes_{\lambda} \R)$. Moreover, $p \times \id$ is a morphism of bisubmersions.
\end{lemma}
\begin{proof}
The composition $\mathcal{O}_{\chi} \times (-\epsilon_{ij},\epsilon_{ij}) \stackrel{(p\times\id)}{\longrightarrow} \mathcal{O}_{i} \times (-\epsilon_{ij},\epsilon_{ij}) \stackrel{\phi_{ij}}{\longrightarrow} \mathbb{T}^2 \rtimes_{\lambda} \R$ is a submersion. The module of right-invariant vector fields of $\mathbb{T}^2 \rtimes_{\lambda} \R$ corresponding to the (``singular'') subalgebroid $\Gamma_c A(\mathbb{T}^2 \rtimes_{\lambda} \R)$ is $\Gamma_c(\ker ds)$. It is straightforward to see that \cite[definition 2.4]{Za22} is satisfied.
\end{proof}

\begin{remarks}\label{rem2}
Lemma \ref{lem5} shows that, for every plot $(\chi,\mathcal{O}_{\chi})$ in $D(G(M,\omega))$, there is $\epsilon > 0$ such that $\mathcal{O}_{\chi} \times (-\epsilon,\epsilon)$ is a bisubmersion of (the ``singular subalgebroid'') $\Gamma_c A(\mathbb{T}^2 \rtimes_{\lambda} \R)$. To explain this, let us look at the dimension of $\mathcal{O}_{\chi}$.
\begin{enumerate}
\item If $dim\mathcal{O}_{\chi} = 1$, there is an $\epsilon > 0$ and a diffeomorphism $\eta : \mathcal{O}_{\chi} \to \R$ such that $h = \iota \circ \eta$, where $\iota$ is the inclusion map in \eqref{seq1}. Note that then the condition $\chi = \pi \circ h$ implies that the map $\chi$ is constant. 

Now define $s, r : \mathcal{O}_{\chi} \times (-\epsilon,\epsilon) \to \mathbb{T}^2$ by $s(x,t) = \iota(t)$ and $r(x,t)=\iota(\phi(x) + t)$. It is straightforward that $(\mathcal{O}_{\chi},s,r)$ is a bisubmersion.

\item If $dim\mathcal{O}_{\chi} \geq 2$, the map $h$ is a submersion because: Identifying $\mathcal{O}_i$ with an open ball in $\R^2$, let $X = \partial_x + \lambda\partial_y$ be the infinitesimal generator of the irrational rotation action and put $F$ the $C^{\infty}(\mathcal{O}_i)$-submodule of $\cX(\mathcal{O}_i)$ generated by the restriction of $X$ to $\mathcal{O}_i$. Now choose a vector field $Y$ of $\mathcal{O}_{\chi}$ such that $dh(Y)=X$ and put $h^{-1}(F)$ the span of $Y$. This is a foliation of $\mathcal{O}_{\chi}$ and by definition $h$ maps leaves of $h^{-1}(F)$ to leaves of $F$. The condition $\chi = \pi \circ h$ implies that $h$ also maps (small) transversals of $h^{-1}(F)$ to (small) transversals of $F$. So, if $L$ is a leaf and $S$ a transversal of $h^{-1}(F)$ at $u$, then $h(L)$ is a leaf and $h(S)$ is a transversal of $F$ at $h(u)$. The decompositions $T_u\mathcal{O}_{\chi} = T_u L \oplus T_u S$ and $T_{h(u)}\mathcal{O}_i = T_{h(u)}(h(L)) \oplus T_{h(u)}(h(S))$ show that $dh$ is surjective.


\item Note that in all cases, there is a submersion $r : \mathcal{O}_{\chi} \times (-\epsilon,\epsilon) \to \mathcal{O}_j$ for some $j \in I$. In the case of lemma \ref{lem5}, $r$ is the composition of the target map of $\mathbb{T}^2 \rtimes_{\lambda} \R$ with the map $(p\times\id) \circ \phi_{ij}$. When $dim\mathcal{O}_{\chi} = 1$ we already defined $r$ and when $dim\mathcal{O}_{\chi} \geq 2$, $r$ is defined as in lemma \ref{lem5}.
\end{enumerate}
\end{remarks}


\subsection{The convolution algebra}\label{sec:convalggpd}

Because of the previous discussion, the convolution algebra of the transformation groupoid $\mathbb{T}^2 \rtimes_{\lambda} \R \gpd \mathbb{T}^2$ can also be constructed using bisubmersions, as we explained in \cite[Appendix A]{Za22}. This construction is more suitable for our purposes, so let us recall it.

\begin{enumerate}
\item Since the range map $r : \mathcal{O}_i \times (-\epsilon_{ij},\epsilon_{ij}) \to  \mathcal{O}_j$ is a submersion we have the exact sequence $$0 \to \ker dr \to T(\mathcal{O}_i \times (-\epsilon_{ij},\epsilon_{ij})) \stackrel{dr}{\longrightarrow} r^{*}(T\mathcal{O}_j) \to 0$$ Whence $\Omega^1(T(\mathcal{O}_i \times (-\epsilon_{ij},\epsilon_{ij})))$ is isomorphic to $\Omega^1(\ker dr) \otimes r^{*}\Omega^1(T\mathcal{O}_j)$.

\item The convolution algebra is defined as $$C^{\infty}_{c}(\mathbb{T}^2 \rtimes_{\lambda} \R) = \frac{\oplus_{i,j}C^{\infty}_c(\mathcal{O}_i \times (-\epsilon_{ij},\epsilon_{ij}); \Omega^1(\ker dr) \otimes r^{*}\Omega^1(\mathcal{O}_j))}{\mathcal{J}}$$ where the subspace $\mathcal{J}$ is defined in a way analogous to the definition of the subspace $\mathcal{I}$ we gave in definition \ref{dfn1}. The exact definition of $\mathcal{J}$ is given right after \cite[Lemma A.4]{Za22}. It is analogous to our lemmata \ref{lem1}, \ref{lem2}. 


\item Given a bisubmersion $r : \mathcal{O}_i \times (-\epsilon_{ij},\epsilon_{ij}) \to \mathcal{O}_j$, put $$Q^{\cA}_{ij} : C^{\infty}_c(\mathcal{O}_i \times (-\epsilon_{ij},\epsilon_{ij}); \Omega^1(\ker dr) \otimes r^{*}\Omega^1(\mathcal{O}_j)) \to C^{\infty}_{c}(\mathbb{T}^2 \rtimes_{\lambda} \R)$$ the obvious quotient map.

\item Let $(U,\phi_U, \mathbb{T}^2 \rtimes \R)$ be a bisubmersion adapted to the atlas $\cA$. Recall from \cite[Appendix A]{Za22} that there is a linear map $Q^{\cA}_U : C^{\infty}_c(U;\Omega^1(U)) \to C^{\infty}_c(\mathbb{T}^2 \rtimes_{\lambda} \R)$ such that:
\begin{itemize}
\item If $(U,\phi_U, \mathbb{T}^2 \rtimes \R)$ is a bisubmersion $r : \mathcal{O}_i \times (-\epsilon_{ij},\epsilon_{ij}) \to \mathcal{O}_j$ then $Q^{\cA}_U = Q^{\cA}_{ij}$.
\item If $p : U' \to U$ is a morphism of bisubmersions which is a submersion, then $Q^{\cA}_{U'} = Q^{\cA}_U \circ p_{!}$.
\end{itemize}

\item Let us recall the definition of the $\ast$-algebra structure of $C^{\infty}_{c}(\mathbb{T}^2 \rtimes_{\lambda} \R)$ from \cite[\text{Appendix A}]{Za22}. Let $(U,\phi_U, \mathbb{T}^2 \rtimes_{\lambda} \R)$ and $(V,\phi_V, \mathbb{T}^2 \rtimes_{\lambda} \R)$ be two bisubmersions adapted to the atlas $\cA$ and $f_U \in C^{\infty}(U;\Omega^1(U))$, $f_V \in C^{\infty}(V;\Omega^1(V))$.
\begin{itemize}
\item The inverse of $(U,\phi_U, \mathbb{T}^2 \rtimes_{\lambda} \R)$ is the bisubmersion $(U^{-1},\phi_{U^{-1}}, \mathbb{T}^2 \rtimes_{\lambda} \R)$ where $U^{-1}=U$ and $\phi_{U^{-1}} = \iota_{\mathbb{T}^2 \rtimes_{\lambda} \R} \circ \phi_{U}$, where $\iota_{\mathbb{T}^2 \rtimes_{\lambda} \R}$ is the inversion map of $\mathbb{T}^2 \rtimes_{\lambda} \R$. It is adapted to $\cA$. Put $f_U^{\ast} = \overline{f}_U \circ \iota_{\mathbb{T}^2 \rtimes_{\lambda} \R}^{-1}$. Involution is defined by $$(Q^{\cA}_U(f_U))^{\ast} = Q^{\cA}_{U^{-1}}(f_U^{\ast})$$
\item The composition $(U \circ V, \phi_U \cdot \phi_V, \mathbb{T}^2 \rtimes_{\lambda} \R)$ is defined by $U \circ V = U \times_{s_U,r_V} V$ and $\phi_U \cdot \phi_V = m_{\mathbb{T}^2 \rtimes_{\lambda} \R} \circ (\phi_U,\phi_V)$, where $m_{\mathbb{T}^2 \rtimes_{\lambda} \R}$ is the groupoid multiplication. It is adapted to $\cA$ as well. Convolution is defined by $$Q^{\cA}_U(f_U)\ast Q^{\cA}_V(f_V) = Q^{\cA}_{V\circ W}(f_U \otimes f_V)$$
\end{itemize}
\end{enumerate}

As we saw in \S \ref{sec:transfgpd} every plot $(\chi,\mathcal{O}_{\chi})$ in $D(G(M,\omega))$, gives rise to a bisubmersion $\mathcal{O}_{\chi} \times (-\epsilon,\epsilon)$ together with a submersion $r : \mathcal{O}_{\chi} \times (-\epsilon,\epsilon) \to \mathcal{O}_{j}$. The exact sequence 
\begin{eqnarray}\label{ext3}
0 \to \ker dr \to T(\mathcal{O}_{\chi} \times (-\epsilon,\epsilon)) \stackrel{dr}{\longrightarrow} r^*(\mathcal{O}_j) \to 0
\end{eqnarray}
induces an isomorphism of $\Omega^1(\mathcal{O}_{\chi}) \otimes \Omega^1(-\epsilon,\epsilon)$ with $\Omega^1(\ker dr) \otimes r^{*}\Omega^1(\mathcal{O}_j)$. Now, every $g \in C^{\infty}_c(\mathcal{O}_{\chi};\Omega^1(\mathcal{O}_{\chi}))$ can be identified with $\hat{g}=g \otimes \ell \in C^{\infty}(\mathcal{O}_{\chi} \times (-\epsilon,\epsilon));\Omega^1(\ker dr) \otimes r^{*}\Omega^1(\mathcal{O}_j))$, where $\ell$ is the Lebesque measure in $(-\epsilon,\epsilon)$. This defines the quotient map $$C^{\infty}(\mathcal{O}_{\chi};\Omega^1(\mathcal{O}_{\chi}))\ni g \mapsto Q_{\chi}^{\cA}(\hat{g}) \in C^{\infty}_c(\mathbb{T}^2 \rtimes_{\lambda}\R)$$

\subsection{The bijection}\label{sec:proof}

For every $i \in I$ consider the plot $(\chi_{ij}, \mathcal{O}_i \times (-\epsilon_{ij},\epsilon_{ij}))$ defined in lemma \ref{lem3} and the submersion $r$. We obtain a linear map $$C^{\infty}_c(\mathcal{O}_i \times (-\epsilon_{ij},\epsilon_{ij}); \Omega^1(\ker dr) \otimes r^{*}\Omega^1(\mathcal{O}_j)) \ni f \mapsto Q_{\chi_{ij}}(f) \in C^{\infty}_c(G(M,\omega))$$ It follows from proposition \ref{prop1} that $Q_{\chi_{ij}}(f)=Q_j(r_!(f))$. Lemma \ref{lem4} shows that the induced map $$\Phi : C^{\infty}_{c}(\mathbb{T}^2 \rtimes_{\lambda} \R) \to C^{\infty}_c(G(M,\omega))$$ is well defined. The fact that $r$ is a submersion shows that $\Phi$ is surjective. The difficulty is to prove the injectivity of $\Phi$. We do this in Proposition \ref{prop2} below. 

\begin{prop}\label{prop2}
The map $\Phi$ is injective.
\end{prop}
\begin{proof}
Let $f$ be an element of $C^{\infty}_c(\mathcal{O}_i \times (-\epsilon_{ij},\epsilon_{ij}); \Omega^1(\ker dr) \otimes r^{*}\Omega^1(\mathcal{O}_j))$ such that $\Phi([f])$ vanishes, where $[f]$ is the class of $f$ in $C^{\infty}_{c}(\mathbb{T}^2 \rtimes_{\lambda} \R))$. This means that there are plots $(\psi,\mathcal{O}_{\psi})$ and $(\chi,\mathcal{O}_{\chi})$ in $D(G(M,\omega))$, a density $g \in C^{\infty}_c(\mathcal{O}_{\psi};\Omega^1\mathcal{O}_{\psi})$ and submersions $p : \mathcal{O}_{\psi} \to \mathcal{O}_j$, $q : \mathcal{O}_{\psi} \to \mathcal{O}_{\chi}$ such that $p_!(g)=r_!(f)$ and $q_!(g)=0$. 

Because of lemma \ref{lem5} and remark \ref{rem2}, there is an $\epsilon > 0$ such that $\mathcal{O}_{\psi} \times (-\epsilon,\epsilon)$, $\mathcal{O}_{\chi} \times (-\epsilon,\epsilon)$ and $\mathcal{O}_{j} \times (-\epsilon,\epsilon)$ are bisubmersions. Put $\hat{p} = p \times \id$, $\hat{q}=q\times \id$. Let $\sigma$ be a local section of the submersion $r : \mathcal{O}_i \times (-\epsilon,\epsilon) \to \mathcal{O}_j$ and put: 
\[
\hat{\sigma} = (\sigma \circ pr_1) \times \pr_2 : \mathcal{O}_j \times (-\epsilon,\epsilon) \to \mathcal{O}_i \times (-\epsilon,\epsilon)
\]
\[
\hat{p} = \hat{\sigma}\circ (p \times \id) :  \mathcal{O}_{\psi} \times (-\epsilon,\epsilon) \to \mathcal{O}_{i} \times (-\epsilon,\epsilon)
\]
\[
\hat{q} = q \times \id : \mathcal{O}_{\psi} \times (-\epsilon,\epsilon) \to \mathcal{O}_{\chi} \times (-\epsilon,\epsilon)
\]
It follows that $f = \hat{p}_!(\hat{g})$ and $\hat{q}_!(\hat{g})=0$, whence the class $[f]$ vanishes.
\end{proof}


The bijection $\Phi$ conveys to $C^{\infty}_c(G(M,\omega))$ the $\ast$-algebra structure of $C^{\infty}_c(\mathbb{T}^2 \rtimes_{\lambda} \R)$. To see this, recall from Lemma \ref{lem5} that every plot $\chi : \mathcal{O}_{\chi} \to G(M,\omega)$ gives rise to a bisubmersion $\mathcal{O}_{\chi} \times (-\epsilon,\epsilon)$ of $\mathbb{T}^2 \rtimes_{\lambda} \R$. The bijection $\Phi$ also conveys the completion of $C^{\infty}_c(\mathbb{T}^2 \rtimes_{\lambda} \R)$. This proves Theorem \ref{thm}.

\appendix

\section{Another $C^{\ast}$-algebra attached to $G(M,\omega)$}\label{app:convolutionalgebraproperties}

As we mentioned in Remark \ref{rmk:staralgebra}, the space $C^{\infty}_c(G(M,\omega))$ admits another $\ast$-algebra structure, arising from the group structure of $\mathbb{T}^2$, rather than the groupoid structure of $\mathbb{T}^2 \rtimes_{\lambda} \R$. For the convenience of the reader, in this appendix we describe this structure. Moreover, we show that the bijection $\Phi$ conveys this structure to the $\ast$-algebra structure associated with the group structure of $\mathbb{T}^2 \rtimes \R$. We will use the notation of lemma \ref{lem2}.

\subsection{Another $\ast$-algebra structure for $C^{\infty}_{c}(G(M,\omega))$}\label{app:anotherstaralgG}

Let us define the convolution first. If $\chi_i : \mathcal{O}_{\chi_i} \to G(M,\omega)$ are plots in $D(G(M,\omega))$, $i=1,2$, then for every $(x,y)$ in $\mathcal{O}_{\chi_1} \times \mathcal{O}_{\chi_2}$ we have $\Omega^1(\mathcal{O}_{\chi_1} \times \mathcal{O}_{\chi_2})_{(x,y)} = \Omega^1(\mathcal{O}_{\chi_1})_x \otimes \Omega^1(\mathcal{O}_{\chi_2})_y$. This identification allows us to define the tensor product of $f_1 \in C^{\infty}_{c}(\mathcal{O}_{\chi_1};\Omega^1 \mathcal{O}_{\chi_1})$ and $f_2\in C^{\infty}_{c}(\mathcal{O}_{\chi_2};\Omega^1 \mathcal{O}_{\chi_2})$ by: $$f_1 \otimes f_2 :  (x,y) \mapsto f_{1}(x) \otimes f_{2}(y) \in C^{\infty}_c(\mathcal{O}_{\chi_1} \times \mathcal{O}_{\chi_2};\Omega^1(\mathcal{O}_{\chi_1} \times \mathcal{O}_{\chi_2}))$$ 
\begin{definition}\label{dfn:convolution}
Given $f_i$ in $C^{\infty}_c(\mathcal{O}_{\chi_i};\Omega^1\mathcal{O}_{\chi_i})$, $i=1,2$ we define $Q_{\chi_1}(f_1) \ast Q_{\chi_2}(f_2)$ to be the class of $(m_{\mathbb{T}^2})_!\left((p_1)_!(g_1) \otimes (p_2)_!(g_2)\right)$.
\end{definition}
\begin{prop}\label{prop:convolution}
The product in definition \ref{dfn:convolution} makes $C^{\infty}_c(G(M,\omega))$ an associative algebra over $\C$. Namely, for every $f_i$ in $C^{\infty}_c(\mathcal{O}_{\chi_i};\Omega^1\mathcal{O}_{\chi_i})$, $i=1,2,3$ and every $\lambda, \mu \in \C$ we have:
\begin{enumerate}
\item $(Q_{\chi_1}(f_1) \ast Q_{\chi_2}(f_2)) \ast Q_{\chi_3}(f_3) = Q_{\chi_1}(f_1) \ast (Q_{\chi_2}(f_2) \ast Q_{\chi_3}(f_3))$
\item $Q_{\chi_1}(f_1) \ast (Q_{\chi_2}(f_2) + Q_{\chi_3}(f_3)) =  Q_{\chi_1}(f_1) \ast Q_{\chi_2}(f_2) + Q_{\chi_1}(f_1) \ast Q_{\chi_2}(f_3)$
\item $(Q_{\chi_1}(f_1) + Q_{\chi_2}(f_2)) \ast Q_{\chi_3}(f_3) = Q_{\chi_1}(f_1) \ast Q_{\chi_2}(f_3) + Q_{\chi_1}(f_2) \ast Q_{\chi_2}(f_3)$
\item $(\lambda Q_{\chi_1}(f_1)) \ast (\mu Q_{\chi_1}(f_1)) = (\lambda\mu)(Q_{\chi_1}(f_1) \ast Q_{\chi_2}(f_2))$
\end{enumerate}
\end{prop}
\begin{proof}
Associativity follows from a straightforward calculation, applying the familiar Fubini theorem. The other algebraic properties follow immediately from the definition.
\end{proof}
Now let us define involution. Consider the inversion map $\iota_{\T^2} : \T^2 \to \T^2$. Let $\chi : \mathcal{O}_{\chi} \to G(M,\omega)$ be a plot in $D(G(M,\omega))$. In view of item (j) in remark \ref{rem:cartprod} put $\mathcal{O}_{\chi}^{-1} = \iota_{\T^2}(\mathcal{O}_{\chi})$ and $\chi^{-1} = \chi \circ \iota_{\T^2}^{-1}$. Then $(\chi^{-1},\mathcal{O}_{\chi}^{-1})$ is obviously an element of $D(G(M,\omega))$. 
\begin{definition}\label{dfn:involution}
For every $f \in C^{\infty}_c(\mathcal{O}_{\chi};\Omega^1(\mathcal{O}_{\chi}))$ we define:
\begin{enumerate}
\item $f^{\ast} = \overline{f} \circ \iota_{\T^2}^{-1}$ in $C^{\infty}_c(\mathcal{O}_{\chi}^{-1};\Omega^1(\mathcal{O}_{\chi}^{-1}))$
\item $(Q_{\chi}(f))^{\ast} = Q_{\chi^{-1}}(f^{\ast})$.
\end{enumerate}
\end{definition}
\begin{prop}\label{prop:involution}
The operation $(Q_{\chi}(f)) \mapsto (Q_{\chi}(f))^{\ast}$ is an involution. It makes $C^{\infty}_c(G(M,\omega))$ a $\ast$-algebra over $\C$. Namely, for every $f_i$ in $C^{\infty}_c(\mathcal{O}_{\chi_i};\Omega^1\mathcal{O}_{\chi_i})$, $i=1,2$ and every $\lambda \in \C$ we have:
\begin{enumerate}
\item $((Q_{\chi}(f_1))^{\ast})^{\ast}=(Q_{\chi}(f_1))$
\item $(Q_{\chi_1}(f_1) \ast Q_{\chi_2}(f_2))^{\ast} = (Q_{\chi_2}(f_2))^{\ast} \ast (Q_{\chi_1}(f_1))^{\ast}$
\item $(Q_{\chi_1}(f_1) + Q_{\chi_2}(f_2))^{\ast} = (Q_{\chi_1}(f_1))^{\ast} + (Q_{\chi_2}(f_2))^{\ast}$
\item $(\lambda(Q_{\chi_1}(f_1)))^{\ast} = \overline{\lambda}(Q_{\chi_1}(f_1))^{\ast}$
\end{enumerate}
\end{prop}
\begin{proof}
Items (c) and (d) are immediate consequences of definition $\ref{dfn:involution}$. As for items (a) and (b):
\begin{enumerate}
\item This follows from the fact that $\iota_{\mathbb{T}^2} \circ \iota_{\mathbb{T}^2} = \id_{\mathbb{T}^2}$ and $\overline{\overline{f}}=f$.
\item In view of lemma \ref{lem2}, definition \ref{dfn:involution} implies that $(Q_{\chi}(f))^{\ast}$ is the class of $(\iota_{\T^2}^{-1})_{!}(p_{!}(\overline{g}))$. Whence $(Q_{\chi_1}(f_1) \ast Q_{\chi_2}(f_2))^{\ast}$ is the class of $$(\iota_{\T^2}^{-1})_{!}((m_{\mathbb{T}^2})_!\left((p_1)_!(\overline{g}_1) \otimes (p_2)_!(\overline{g}_2)\right)) = (\iota_{\T^2}^{-1} \circ m_{\mathbb{T}^2})_{!}\left((p_1)_!(\overline{g}_1) \otimes (p_2)_!(\overline{g}_2)\right))$$ We have $\iota_{\T^2}^{-1}=\iota_{T^2}$, so $\iota_{\T^2}^{-1} \circ m_{\mathbb{T}^2} = \iota_{\T^2} \circ m_{\mathbb{T}^2} = m_{\mathbb{T}^2} \circ (\iota_{\T^2} \times \iota_{\T^2}) = m_{\T^2} \circ (\iota_{\T^2}^{-1} \times \iota_{\T^2}^{-1})$.  Whence the right-hand side is $$(m_{\mathbb{T}^2})_{!}\left((\iota_{\T^2}^{-1})_{!}((p_2)_!(\overline{g}_2)) \otimes (\iota_{\T^2}^{-1})_{!}((p_1)_!(\overline{g}_1))\right)$$ The latter is precisely $(Q_{\chi_2}(f_2))^{\ast} \ast (Q_{\chi_1}(f_1))^{\ast}$.
\end{enumerate}
\end{proof}

\subsection{Another $\ast$-algebra structure for $C^{\infty}(\mathbb{T}^2 \rtimes_{\lambda} \R)$}\label{app:anotherstaralggpd}

This $\ast$-algebra structure arises from viewing $\mathbb{T}^2 \rtimes_{\lambda} \R$ as a group. This group structure is the cartesian group structure arising from the groups $\mathbb{T}^2 = S^1 \times S^1$ and $(\R,+)$. Explicity, for $i=1,2$ consider $g_i = (e^{2\pi i \theta_i}, e^{2\pi i \eta_i}) \in \mathbb{T}^2$ and $t_i \in \R$. Then: $$(g_1,t_1)\cdot (g_2,t_2) = ((e^{2\pi i (\theta_1 + \theta_2)}, e^{2\pi i (\eta_1 + \eta_2)}),t_1 + t_2)$$ and $$g_1^{-1} = ((e^{2\pi i (-\theta_i)}, e^{2\pi i (-\eta_i})),-t_1)$$ Let us demote $m : (\mathbb{T}^2 \rtimes_{\lambda} \R) \times (\mathbb{T}^2 \rtimes_{\lambda} \R) \to \mathbb{T}^2 \rtimes_{\lambda} \R$ and $inv : \mathbb{T}^2 \rtimes_{\lambda} \R \to \mathbb{T}^2 \rtimes_{\lambda} \R$ the above multiplication and inversion maps respectively. 

Now let us define the new convolution and involution in $C^{\infty}_c(\mathbb{T}^2 \rtimes_{\lambda} \R)$. Let $(U,\phi_U, \mathbb{T}^2 \rtimes_{\lambda} \R)$ and $(V,\phi_V, \mathbb{T}^2 \rtimes_{\lambda} \R)$ be two bisubmersions adapted to the atlas $\cA$ and $f_U \in C^{\infty}(U;\Omega^1(U))$, $f_V \in C^{\infty}(V;\Omega^1(V))$.
\begin{enumerate}
\item The inverse of $(U,\phi_U, \mathbb{T}^2 \rtimes_{\lambda} \R)$ is the bisubmersion $(U^{-1},\phi_{U^{-1}}, \mathbb{T}^2 \rtimes_{\lambda} \R)$ where $U^{-1}=U$ and $\phi_{U^{-1}} = inv \circ \phi_{U}$. It is adapted to $\cA$. Put $f_U^{\ast} = \overline{f}_U \circ inv^{-1}$. Involution is defined by $$(Q^{\cA}_U(f_U))^{\ast} = Q^{\cA}_{U^{-1}}(f_U^{\ast})$$

\item The composition $(U \circ V, \phi_U \cdot \phi_V, \mathbb{T}^2 \rtimes_{\lambda} \R)$ is defined by $U \circ V = U \times_{s_U,r_V} V$ and $\phi_U \cdot \phi_V = m \circ (\phi_U,\phi_V)$. It is adapted to $\cA$ as well. Convolution is defined by $$Q^{\cA}_U(f_U)\ast Q^{\cA}_V(f_V) = Q^{\cA}_{V\circ W}(m_{!}(f_U \otimes f_V))$$
\end{enumerate}

he above involution and involution make $C^{\infty}_c(\mathbb{T}^2 \rtimes_{\lambda} \R)$ a $\ast$-algebra. The proof of this is as in \S \ref{app:anotherstaralgG}, so we omit it.

\subsection{$\Phi$ is a $\ast$-isomorphism}\label{app:Phistariso}

\begin{prop}
The bijection $\Phi$ is a $*$-homomorphism for the $\ast$-algebra structures defined in \S \ref{app:anotherstaralgG} and \ref{app:anotherstaralggpd}.
\end{prop}
\begin{proof}
First we show that $\Phi$ respects involution. Choose a section $f$ in $C^{\infty}_c(\mathcal{O}_i \times (-\epsilon_{ij},\epsilon_{ij}); \Omega^1(\ker dr) \otimes r^{*}\Omega^1(\mathcal{O}_j))$. Put $U = \mathcal{O}_i \times (-\epsilon_{ij},\epsilon_{ij})$ and $\chi_{ij} = \pi|_{\mathcal{O}_i} \circ r$. Recall that the map $r$ is really the restriction of the target map of the groupoid $\mathbb{T}^2 \rtimes_{\lambda} \R$. The inversion map $\iota_{\mathbb{T}^2 \rtimes_{\lambda} \R}$ satisfies $r \circ \iota_{\mathbb{T}^2 \rtimes_{\lambda} \R} = s$, where $s$ is the source map of $\mathbb{T}^2 \rtimes_{\lambda} \R$. It follows that $U^{-1} = \mathcal{O}_j \times (-\epsilon_{ij},\epsilon_{ij})$ and $r : \mathcal{O}_j \times (-\epsilon_{ij},\epsilon_{ij}) \to \mathcal{O}_i$. In particular, the following diagram commutes:
\[
\begin{CD}
\mathcal{O}_i \times (-\epsilon_{ij},\epsilon_{ij})             @>r>>   \mathcal{O}_j \\ 
@V\iota_{\mathbb{T} \rtimes_{\lambda} \R}VV                @VV\iota_{\mathbb{T}}V \\
\mathcal{O}_j \times (-\epsilon_{ij},\epsilon_{ij})            @>r>>   \mathcal{O}_i
\end{CD}
\]
Hence we have $(Q^{\cA}_{ij}(f))^{\ast} = Q^{\cA}_{ji}(\overline{f}\circ \iota^{-1}_{\mathbb{T}^2 \rtimes_{\lambda} \R})$ and $(Q_j(r_{!}(f)))^{\ast} = Q_{i}(\overline{r_!(f)}\circ \iota_{\mathbb{T}^2}^{-1})$. Therefore, we find: 
\begin{equation*}
\begin{split}
\Phi((Q^{\cA}_{ij}(f))^{\ast}) & =  \Phi(Q^{\cA}_{ji}(\overline{f}\circ\iota^{-1}_{\mathbb{T}^2 \rtimes_{\lambda} \R})) = Q_{i}(r_{!}(\overline{f} \circ \iota^{-1}_{\mathbb{T}^2 \rtimes_{\lambda} \R})) \\ & = Q_{i}(r_{!}(\overline{f})\circ \iota^{-1}_{\mathbb{T}^2}) = Q_{i}(\overline{r_{!}(f)}\circ \iota^{-1}_{\mathbb{T}^2}) \\ & = (Q_i(r_{!}f))^{\ast} = (\Phi(Q^{\cA}_{ij}(f)))^{\ast}
\end{split}
\end{equation*}
The fact that $\Phi$ respects convolution follows from the definitions of convolution in Definition \ref{dfn:convolution} and \S \ref{sec:anotherstaralggpd} and the commutativity of the following diagram: 
\[
\begin{CD}
(\mathbb{T}^2 \rtimes_{\lambda} \R) \times (\mathbb{T}^2 \rtimes_{\lambda} \R)             @>m>>   \mathbb{T}^2 \rtimes_{\lambda} \R \\ 
@Vr\times rVV                @VVrV \\
\mathbb{T}^2 \times \mathbb{T}^2            @>m_{\mathbb{T}^2}>>   \mathbb{T}^2
\end{CD}
\]
\end{proof}

\bibliographystyle{halpha} 



\end{document}